\documentclass[notitlepage,11pt,reqno]{amsart}

\usepackage{amsmath,amsfonts,amsopn,amssymb,esint,nicefrac}
\usepackage[final]{hyperref}
\usepackage[T1]{fontenc}
\usepackage{esint}

\usepackage[utf8]{inputenc}
\usepackage{apptools}
\AtAppendix{\counterwithin{lemma}{section}} 
\usepackage[margin=1in]{geometry}
\allowdisplaybreaks
\newcommand{\stkout}[1]{\ifmmode\text{\sout{\ensuremath{#1}}}\else\sout{#1}\fi}

\newcommand{\Rn}{\mathbb{R}^N}
 
  \newcommand{\eps} {\varepsilon_0}

\usepackage{graphicx,enumitem,dsfont,upgreek}
 \usepackage[dvips]{epsfig}
 \usepackage[mathscr]{eucal}
\usepackage{amscd}
\usepackage{amssymb}
\usepackage{amsthm}
\usepackage{amsmath}
\usepackage{latexsym}
\usepackage{dsfont}
\usepackage{upref}
\usepackage{hyperref}

\usepackage{enumitem}

\usepackage{color}
\theoremstyle{plain}

\newtheorem{thm}{Theorem}[section]
\theoremstyle{plain}
\newtheorem{lemma}[thm]{Lemma}

\newtheorem{cor}[thm]{Corollary}

\theoremstyle{definition}
\newtheorem{defi}{Definition}[section]
\newtheorem{rem}{Remark}[section]
\newtheorem*{maintheorem*}{Main Theorem}
\newtheorem*{maincorollary*}{Main Corollary}
{%
\setcounter{enumi}{0}

\begin{enumerate}}%
{\end{enumerate} }

{%
\setcounter{enumi}{0}

\begin{enumerate}}%
{\end{enumerate} }

\newcommand{\abs}[1]{\ensuremath{\left|#1\right|}}
\newcommand{\Om}{\ensuremath{\Omega}}

\definecolor{darkgreen}{rgb}{0.09, 0.65, 0.27}
\newcommand{\cred}{\color{red}}

\newcommand{\cb}{\color{blue}}

\newcommand{\dist}{\rm dist}
\newcommand{\na}{\nabla}

\newcommand{\R}{\ensuremath{\mathbb{R}}}

\newcommand{\Grad}{\mathrm{\nabla}}

\newcommand{\dy}{\ensuremath{\, {\rm d}y}}
\newcommand{\dz}{\ensuremath{\, {\rm d}z}}

\numberwithin{equation}{section} \allowdisplaybreaks

\usepackage{cancel,pdfsync}

\title[Liouville properties with gradient nonlinearity]{Liouville results for supersolutions of fractional 
$p$-Laplacian equations with gradient nonlinearities}

\begin{document}

\author{Mousomi Bhakta, Anup Biswas and Aniket Sen}
\address{Department of mathematics, Indian Institute of Science Education and Research Pune, Dr.\
Homi Bhabha Road, Pune 411008, India}
\email{mousomi@iiserpune.ac.in, \, anup@iiserpune.ac.in,\, 
aniket.sen@students.iiserpune.ac.in}

\begin{abstract}
We prove that any nonnegative viscosity solution of the inequality
$$(-\Delta_p)^s u \geq u^{t} |\na u|^{m}\quad \text{ in }\; \Rn,\; N\geq 2,$$
must be constant. This result holds for parameters $p\in (1, \infty), s\in (0, 1)$, 
$t, m\geq 0$, satisfying
$$t (N-sp) + m(N-(sp-p+1)) < N(p-1),$$
with the additional condition that either $m\leq p-1$ if $p-1<sp$, or $m<sp$ if $p-1\geq sp$.

\end{abstract}


\maketitle

\medskip

\noindent

\section{Introduction}
This paper is devoted to establishing a Liouville theorem for inequalities involving the fractional $p$-Laplacian with a gradient-dependent nonlinearity. Our principal result is the following.
\begin{thm}\label{gradnon} 
Suppose that $0<s<1<p<\infty$, $N\geq 2$, $t\geq 0$ and
\begin{align}\label{ineq:m}
 0\leq m\begin{cases}\leq p-1  \quad\mbox{if}\quad \frac{sp}{p-1} >1,\\
<sp \quad\mbox{if}\quad \frac{sp}{p-1} \leq 1.
\end{cases}
\end{align}
Moreover, let $N>sp$ and
\begin{equation}\label{algineq}
t (N-sp) + m\bigl(N-(sp-p+1)\bigr) < N(p-1).
\end{equation} 
Then, any nonnegative viscosity solution to
\begin{equation}\label{E1.1A}
(-\Delta_p)^s u \geq u^{t} |\na u|^{m}\quad  \text{ in }\; \Rn
\end{equation}
is a constant. 
\end{thm}
Here $(-\Delta_p)^s$ denotes the fractional $p$-Laplacian operator defined by
$$(-\Delta_p)^s u(x)={\rm PV}\int_{\Rn} |u(x)-u(y)|^{p-2}(u(x)-u(y))\frac{\dy}{|x-y|^{N+sp}}.$$
\begin{rem}
  For $N\leq sp$, it is known from \cite[Theorem~1.2]{DPQ25} that any nonnegative viscosity solution to $(-\Delta_p)^s u\geq 0$ in $\Rn$ is constant. 
\end{rem}

\begin{rem}
The case $m=0$ which corresponds to supersolutions of $p$-fractional Lane-Emden equation,  Theorem~\ref{gradnon} has been very  recently proved by Del Pezzo and Quaas \cite[Theorem 1.3]{DPQ25}. The proof of \cite{DPQ25} relies on a technique based on maximum principle, originally introduced in \cite{CL00}. But, for $t>0$, the method of \cite{DPQ25} cannot be applied to our model due to the presence of gradient nonlinearity. In this article, we overcome this obstacle by developing a new, self-contained approach centered on an iterative argument (see Lemma~\ref{L1.3}).
\end{rem}

Liouville property for the differential operators are widely studied not only because of its applications in regularity theory and blow-up analysis, but also due to its intrinsic theoretical interest. For $s=1$, there are quite a few works
dealing with Liouville results for inequalities of type 
\eqref{E1.1A}. See, for instance, \cite{ACQ,BBF25,BGV19,BPGQ16,CM97,CHZ22,F09,MP01,MP99}. 
It is interesting to observe that these problems belong to the class of Liouville-type results asserting that bounded-above subharmonic functions must be constant. This class of Liouville theorems has been investigated in a much broader context and established using a variety of methods, Hadamard three circle
theorem in \cite{PW-book},
Caccioppoli estimates in \cite{Moon}, stochastically incompleteness property in \cite{Gri99}, and the divergence theorem in \cite{RS01} (see also \cite{PRS08} for several associated equivalence relations).
The recent survey \cite{CG23} provides a comprehensive overview of Liouville theorems with gradient nonlinearities for classical elliptic operators (that is, $s=1$). 

In contrast, for the nonlocal setting ($s \in (0,1)$), the literature is less developed, with most existing results, such as those in \cite{B19, DKK08, DPQ25, QX16}, addressing the case $m = 0$ (i.e., without a gradient term). The case involving a gradient nonlinearity has remained open and is notably challenging, as highlighted by the open problems presented in \cite{CG23}. 
A recent innovative approach based on the Ishii-Lions method was introduced in \cite{BQT} to tackle such problems but for equations. To the best of our knowledge, Theorem \ref{gradnon} provides the first Liouville results for inequalities involving fractional $p$-Laplacian in the presence of a gradient nonlinearity. 

{The primary difficulty stems from the inapplicability of techniques that are standard in the local setting. Let us briefly review some well-established methods for the case $s=1$. The authors in \cite{BPGQ16} construct suitable radial solutions by solving certain ordinary differential equations and then apply the maximum principle to establish the Liouville property. In \cite{CHZ22}, a transformation of the form \(v = u^b\) is employed to eliminate the zero-order term (corresponding to \(t=0\)), followed by the nonlinear capacity method of Mitidieri (see \cite{BGV19,MP01}). The nonlinear capacity approach of \cite{MP99,MP01} uses test functions of the form $u^{-d}\xi^k$, for a cut-off function $\xi$ and $d>0$, and derives integral estimates. For the case $t=0$, \cite{Gof23} applies the divergence theorem, in the spirit of \cite{RS01}, to establish the Liouville property. Another approach, based on the strong maximum principle and Lyapunov functions, was introduced in \cite{Bardi2016} for Pucci-type subadditive operators; see also \cite{Bardi2022} for an extension to the Heisenberg group.

These methods rely heavily on the local nature of the operator and do not readily extend to the nonlocal framework of the fractional $p$-Laplacian, making Theorem~\ref{gradnon} a genuinely nontrivial extension. Moreover, our operators do not satisfy the structural conditions required in \cite{Bardi2016,Bardi2022}. We also note that condition~\eqref{algineq} is the natural nonlocal analogue of the well-known subcritical condition in the local case; see \cite[Theorem~15.1]{MP01}, as well as \cite[Theorem~2.1]{BGV19} and \cite[Theorem~1.2]{CHZ22}.

At this point, we briefly describe the strategy of the proof. Using the maximum principle, we first show in Lemma~\ref{L1.2} that for any 
$\sigma_1 \in \big(\frac{N-sp}{p-1}, \frac{N}{p-1}\big)$, the solution satisfies
$$
u(x) \gtrsim \min\{1, |x|^{-\sigma_1}\}.
$$
Substituting this bound into \eqref{E1.1A}, we obtain
$$
(-\Delta_p)^s u(x) \gtrsim |x|^{-t\sigma_1} |\nabla u|^{m}.
$$
We then construct a strict subsolution of the form $\varphi_{\sigma_2}(x) = c |x|^{-\sigma_2}$ for $|x| \ge 1$, with a suitable positive constant $c$ and some $0 < \sigma_2 < \sigma_1$. A further application of the comparison principle yields
\[
u(x) \gtrsim \min\{1, |x|^{-\sigma_2}\}.
\]
Using \eqref{ineq:m}--\eqref{algineq}, we show that this procedure can be repeated with $\sigma_2$, leading to a $\sigma_3$ with $0 < \sigma_3 < \sigma_2$ such that
\[
u(x) \gtrsim \min\{1, |x|^{-\sigma_3}\}.
\]
Iterating this scheme, we eventually obtain
$$
u(x) \gtrsim \min\{1, |x|^{-\theta}\}
$$
for any small $\theta > 0$. A final application of the maximum principle then yields the desired result.

We conclude the introduction with the following remark, which highlights some directions for future investigation.
\begin{rem}
The condition \eqref{ineq:m} makes Theorem~\ref{gradnon} somewhat more restrictive than the corresponding results known for the case $s=1$, which are obtained under the sole assumption \eqref{algineq}. This additional requirement appears to stem from a technical constraint arising in the proof of our key Lemma~\ref{L1.3}, where the condition is needed to carry out the iteration argument. It is plausible that the conclusion of Theorem~\ref{gradnon} remains valid under the weaker assumption \eqref{algineq} alone; however, establishing this presently remains an open problem. Furthermore, our method does not address the borderline case of equality in \eqref{algineq}, namely
\[
t (N-sp) + m\bigl(N-(sp-p+1)\bigr) = N(p-1).
\]
We note that, in the local case $s=1$, the Liouville theorem continues to hold under this condition provided $m>p-1$; see \cite{BBF25,Gof23}.

In the case $t=0$, equation \eqref{E1.1A} reduces to $(-\Delta_p)^s u \ge |\nabla u|^m$, and it is natural to expect a Liouville-type result without any lower bound assumption on $u$ (compare with \cite{BBF25,Gof23} for $s=1$). However, our present approach relies on a one-sided bound on $u$, and we do not know at this stage how to remove this assumption.
\end{rem}
}


The rest of the article is organized as follows. Section~\ref{S-prim} collects necessary preliminary results. And 
Section~\ref{S-main} is devoted to the proof of our key iterative lemma and the completion of the proof of Theorem~\ref{gradnon}.

%

\section{Preliminary}\label{S-prim}
In this section, we provide some preliminary  estimates which will be used in the proof of Theorem~\ref{gradnon}.  The solution to \eqref{E1.1A} is understood in the viscosity sense, which we recall from  \cite{KKL19}. 
First, we recall some notation from 
\cite{KKL19}.  Since, as noticed in \cite{KKL19}, the operator $(-\Delta_p)^s$ may not be classically defined for all $C^2$ functions, we must restrict our consideration to a suitable subclass of test functions when defining viscosity solutions.
Given an open set $D$, we denote by $C^2_\eta(D)$, a subset of $C^2(D)$, defined as
\begin{equation}\label{Ceta}
C^2_\eta(D)=\left\{\phi\in C^2(D)\; :\; \sup_{x\in D\setminus N_\phi}\left[\frac{\min\{d_\phi(x), 1\}^{\eta-1}}{|\nabla\phi(x)|} +
\frac{|D^2\phi(x)|}{(d_\phi(x))^{\eta-2}}\right]<\infty\right\},
\end{equation}
where
$$ 
d_\phi(x)={\dist}(x, N_\phi)\quad \text{and}\quad N_\phi=\{x\in D\; :\; \nabla\phi(x)=0\}.$$

The above restricted class of test functions becomes necessary to define $(-\Delta_p)^s$ in the classical sense in the singular case, that is, for
$p\leq \frac{2}{2-s}$.  We also denote by
$$L^{p-1}_{sp}(\Rn)=\{w\in L^{p-1}_{\rm loc}(\Rn)\; :\; 
\int_{\Rn}\frac{|w(z)|^{p-1}}{1+|z|^{N+sp}}\dz<\infty\}. $$
Now we are ready to define the viscosity solution from \cite[Definition~3]{KKL19}. 

\begin{defi}\label{Def2.1}
Let $\Omega$ be a nonempty open set.
A function $u:\Rn\to \R$ is a viscosity supersolution to 
$(-\Delta_p)^su = f(x, u, \na u) $ in $\Omega$,
where $f$ is a continuous function, if it satisfies the following
\begin{itemize}
\item[(i)] $u$ is lower semicontinuous in $\bar\Omega$.
\item[(ii)] If $\varphi\in C^2(B_r(x_0))$ for some $B_r(x_0)\subset \Omega$ satisfies $\varphi(x_0)=u(x_0)$,
$\varphi\leq u$ in $B_r(x_0)$ and one of the following holds
\begin{itemize}
\item[(a)]  $p>\frac{2}{2-s}$ or $\nabla\varphi(x_0)\neq 0$,

\item[(b)] $p\leq \frac{2}{2-s}$ and $\nabla\varphi(x_0)= 0$ is such that $x_0$ is an isolated critical point of $\varphi$, and
$\varphi\in C^2_\eta(B_r(x_0))$ for some $\eta>\frac{sp}{p-1}$,
\end{itemize}
then we have $(-\Delta_p)^s \varphi_r(x_0)\geq f(x_0,  \varphi(x_0), \na \varphi(x_0))$, where
\[
\varphi_r(x)=\left\{\begin{array}{ll}
\varphi(x) & \text{for}\; x\in B_r(x_0),
\\[2mm]
u(x) & \text{otherwise}.
\end{array}
\right.
\]

\item[(iii)] We have $u_-\in L^{p-1}_{sp}(\Rn)$.
\end{itemize}

We say $u$ is a viscosity subsolution in $\Omega$, if $-u$ is a viscosity supersolution in $\Omega$. Furthermore, a 
viscosity solution of $(-\Delta_p)^s u= f$ in $\Omega$ is both sub and super solution in $\Omega$. 
\end{defi}

\begin{rem}
In this article, we employ the concept of viscosity solutions, as they require minimal regularity assumptions on the solution. However, if we assume that $u$ is of class $C^1$, it can be readily shown that any $C^1$ weak solution is also a viscosity solution. Consequently, our Theorem~\ref{gradnon} remains applicable in this setting.

We note that for similar problems involving elliptic operators with gradient nonlinearities, it is standard practice to impose $C^1$ regularity; see \cite{BBF25,BGV19,CHZ22}. In contrast, establishing such regularity for the fractional $p$-Laplacian remains a challenging open problem. For the current state of this problem, we refer the reader to \cite{BS25,BT25,BLS,GL24} and the references therein. Notably, the recent work \cite{GJS25} establishes $C^{1,\alpha}$ estimates for fractional $p$-harmonic functions, but only for $p\in [2,\frac{2}{1-s})$.
\end{rem}

We start with a comparison principle, suitable for our purpose. 
\begin{lemma}\label{comparison}
Let $\Omega$ be a bounded domain of $\Rn$. Let 
$u \in C(\bar\Omega)\cap L^{p-1}_{sp}(\Rn)$ be a viscosity supersolution to    
$$(-\Delta_p)^s u = f(x, \Grad u)\quad  \text{ in }\; \Omega,$$ and $v \in C(\bar\Omega)\cap L^{p-1}_{sp}(\Rn)$ be solution to    
$$(-\Delta_p)^s v < f(x, \Grad v)\quad  \text{ in }\; \Omega.$$ 
In addition, also assume that one of $u, v$ is in $C^2(\Omega)$
and its gradient does not vanish in $\Omega$.
If $u \geq v$ in $\Omega^c$, then $u \geq v$ in $\Rn$.
\end{lemma} 
\begin{proof}
Without any loss of generality, we assume that $v\in C^2(\Omega)$
and $|\na v|\neq 0$ in $\Omega$.
Suppose, on the contrary, that $u\ngeq v$ in $\Rn$. Then, by our
hypothesis, there exists $x_0\in\Omega$ so that
$$\rho_0:=v(x_0)-u(x_0)=\max_{\Rn} (v-u)>0.$$
Define,
$\varphi(x) := v(x) - \rho_0$. Therefore, $\varphi(x)\leq u(x)$, $\varphi(x_0)=u(x_0)$ and $\na\varphi (x_0)=\na v(x_0) \neq 0$. Thus, applying the definition of viscosity supersolution
we obtain $(-\Delta_p)^s\varphi_r(x_0)\geq f(x_0, \na v(x_0))$, where $\varphi_r$ is given by Definition ~\ref{Def2.1} and
$B_r(x_0)\Subset\Omega$. Since
$\varphi_r(x)-\varphi_r(x_0)\geq v(x)-v(x_0)$ for $x\in\Rn$, we obtain using monotonicity of the map $t\mapsto |t|^{p-2}t$ that
$$f(x_0, \na v(x_0))\leq (-\Delta_p)^s\varphi_r(x_0)
\leq (-\Delta_p)^s v(x_0)< f(x_0, \na v(x_0)),$$
which is a contradiction. Hence, we must have $u\geq v$ in $\Rn$,
completing the proof.
\end{proof}
Next result is the strong maximum principle in viscosity solution setting. For an analogous result for weak supersolutions
we refer to \cite{DPQ17}.

\begin{lemma}\label{maxpr}
Let $\Omega$ be an open set in $\Rn$.
Let u be a nontrivial nonnegative solution to $(-\Delta_p)^s u \geq 0$ in  $\Omega$. Then $u>0$ in  $\Omega$.
\end{lemma}
\begin{proof}
Suppose, on the contrary, that $u(x_0)=0$ for some $x_0 \in \Omega$. Let $\varphi(x) = -|x-x_0|^\eta$ for some large 
$\eta>\max\{\frac{sp}{p-1}, 2\}$, so that $\varphi \in C^2_\eta(B_{r}(x_0))$ (see \eqref{Ceta}).  Then, from \cite[Lemma~3.6 and 3.7]{KKL19}, we have
$$
\lim_{r\to 0}\left|{\rm PV}\int_{B_r(x_0)} J_p(\varphi(x_0)-\varphi(y)) \frac{\dy}{|x_0-y|^{N+sp}}\right|=0,
$$
where $J_p(t)=|t|^{p-2}t$. Since $u$ is nontrivial and nonnegative,
we can choose $r$ small enough so that
\begin{equation}\label{EL2.2A}
(-\Delta_p)^s\varphi_r(x_0)={\rm PV}\int_{B_r(x_0)} J_p(\varphi(x_0)-\varphi(y)) \frac{\dy}{|x_0-y|^{N+sp}} + \int_{B^c_r(x_0)} J_p(-u(y)) \frac{\dy}{|x_0-y|^{N+sp}}<0.
\end{equation}
We note that $\varphi(x_0) =u (x_0) =0$, $\varphi(x) \leq u(x)$ in  $B_{r}(x_0) $. Thus,
by Definition~\ref{Def2.1}, we have $(-\Delta_p)^s\varphi_r(x_0)\geq 0$, which contradicts \eqref{EL2.2A}.
Hence we have the proof.
\end{proof}

Let us now recall the following result from \cite[Theorem~1.1]{DPQ25}.
\begin{thm}\label{DPQ1}
Let $N \geq 2$, $s\in (0, 1)$ and $p > 1$.  Also, let $sp \neq N$. Then $\phi_{\theta} (x) = |x|^{-\theta}$, where $\theta \in \left( -\frac{sp}{p-1}, \frac{N}{p-1} \right)$, is a viscosity solution to
$$
(-\Delta_p)^{s} \phi_{\theta} (x) = C(\theta)\, |x|^{-\theta(p-1) -sp} \quad \text{in}\;\;  \Rn \setminus \{0\},$$
where $C(\theta)$ is a constant satisfying
\[C(\theta)
\begin{cases}
=0 \quad\text{ if } \theta = 0 \text { or } \theta = \frac{N-sp}{p-1},
\\
> 0 \quad \text{ if } \min \left\{-\frac{N-sp}{p-1}, 0 \right\} < -\theta < \max \left\{ -\frac{N-sp}{p-1}, 0\right\},
\\
< 0 \quad \text{ otherwise. }
\end{cases}
\]
\end{thm}
We remark that the result \cite[Theorem~1.1]{DPQ25} is stated for weak solutions, but due to the equivalence between continuous weak solution and viscosity solution \cite[Theorem~1.3]{BM21}, the above result holds. Furthermore, since $\na \phi_\theta\neq 0$ in
$\Rn\setminus \{0\}$ and $\phi_\theta$ is $C^2$ for $x\neq 0$,
we see that $\phi_\theta$ is also a classical solution.

\begin{rem}  Note that for our Theorem~\ref{gradnon}, we assume 
$N>sp\Rightarrow - \frac{N-sp}{p-1} <0$. Thus, for
$\theta\in \left(0, \frac{N-sp}{p-1} \right)$, we have $C(\theta) >0$.
\end{rem}


%
We also need the following result.
\begin{lemma}\label{DPQ3}
Let $N>sp$ and $\theta\in (0, \frac{N}{p-1})$. For 
$\varepsilon_0 \in (0,\frac{1}{2}]$, we define
$$\varphi_{\theta} (x) :=
\begin{cases}
|x|^{-\theta} \text{ if } |x| \geq \varepsilon_0,
\\[2mm]
\varepsilon_0^{-\theta} \text{\quad if } |x| < \varepsilon_0.
\end{cases}
$$
Then,  for some $\varepsilon_0=\varepsilon_0(\theta)\in (0, \frac{1}{2}]$, we have the following: 
\begin{itemize}
\item[(i)] For $\theta \in \left(0, \frac{N-sp}{p-1} \right]$,
we have
$\varphi_\theta$ a classical solution to 
$$
(-\Delta_p)^s \varphi_{\theta} (x) \leq C(\theta, \varepsilon_0)\, |x|^{-\theta(p-1) -sp}\quad  \text{ in }\;\; \Rn \setminus B_1,$$ and $C(\theta, \varepsilon_0) >0$.

\item[(ii)] For $\theta \in \left( \frac{N-sp}{p-1}, \frac{N}{p-1} \right)$, $\varphi_\theta$ is a classical solution to 
$$(-\Delta_p)^s \varphi_{\theta} (x) < 0 \quad \text{ in }\;\; \Rn \setminus B_1.$$
\end{itemize}
\end{lemma}

\begin{proof}
Let $x\in \Rn \setminus B_1$ and we compute using Theorem~\ref{DPQ1}
\begin{align*}
(-\Delta_p)^s \varphi_\theta(x)
&= (-\Delta_p)^s |x|^{-\theta} + \int_{B_{\varepsilon_0}} | \varphi_{\theta}(x) - \varphi_{\theta}(y)|^{p-2} (\varphi_{\theta}(x) -\varphi_{\theta}(y)) \frac{\dy}{|x-y|^{N+sp}} 
\\
&\qquad- \int_{B_{\varepsilon_0}} \abs{|x|^{-\theta} -|y|^{-\theta} }^{p-2} (|x|^{-\theta}- |y|^{-\theta}) \frac{\dy}{|x-y|^{N+sp}}
\\
&= C(\theta)\, |x|^{-\theta(p-1) -sp} -  \int_{B_{\varepsilon_0}} | \eps^{-\theta} - |x|^{-\theta}|^{p-1}  \frac{\dy}{|x-y|^{N+sp}} 
\\
&\qquad+ \int_{B_{\varepsilon_0}} \abs{|y|^{-\theta} -|x|^{-\theta} }^{p-1}  \frac{\dy}{|x-y|^{N+sp}}
\\
&\leq C(\theta)\, |x|^{-\theta(p-1) -sp} + \int_{B_{\varepsilon_0}} \abs{|y|^{-\theta} -|x|^{-\theta} }^{p-1}  \frac{\dy}{|x-y|^{N+sp}}
\\
& =C(\theta)\, |x|^{-\theta(p-1) -sp} + |x|^{-\theta(p-1)} \int_{B_{\varepsilon_0}} \abs{|{y}/{|x|}|^{-\theta} -1 }^{p-1}  \frac{\dy}{
|x|^{N +sp}\abs{\frac{x}{|x|}-\frac{y}{|x|}}^{N+sp}}
\\
& =C(\theta)\, |x|^{-\theta(p-1) -sp} + |x|^{-\theta(p-1)-sp} \int_{B_{\frac{\varepsilon_0}{|x|}}} \abs{|z|^{-\theta} -1 }^{p-1}  \frac{\dz}{
\abs{x/|x|-z}^{N+sp}}
\\
& \leq C(\theta)\, |x|^{-\theta(p-1) -sp} + |x|^{-\theta(p-1)-sp} \int_{B_{\frac{\varepsilon_0}{|x|}}} |z|^{-\theta(p-1)}  \frac{\dz}{
\abs{1-\frac{\varepsilon_0}{|x|}}^{N+sp}}
\\
& =C(\theta)\, |x|^{-\theta(p-1) -sp} + \frac{|x|^{-\theta(p-1)-sp}}{{
\abs{1-\frac{\varepsilon_0}{|x|}}^{N+sp}}} \int_{B_{\frac{\varepsilon_0}{|x|}}} |z|^{-\theta(p-1)}  \dz
\\
& \leq C(\theta)\, |x|^{-\theta(p-1) -sp} + \frac{|x|^{-\theta(p-1)-sp}}{{
\abs{1-\varepsilon_0}^{N+sp}}} \int_{B_{\varepsilon_0}} |z|^{-\theta(p-1)}  \dz.
\end{align*}
Note that for $\theta(p-1)< N$, we have $z\mapsto |z|^{-\theta(p-1)}$ integrable around $0$.
For $\theta \in \left(0, \frac{N-sp}{p-1} \right]$, we have $C(\theta) \geq 0$ and (i) follows by letting $\varepsilon_0=\frac{1}{2}$.

For $\theta \in \left( \frac{N-sp}{p-1}, \frac{N}{p-1} \right)$,
we have $C(\theta) <0$ by Theorem~\ref{DPQ1}, and therefore,  we 
can choose $\varepsilon_0$ small enough, depending on $C(\theta)$,
so that (ii) holds.

\end{proof}


\section{Proof of Theorem~\ref{gradnon}}\label{S-main}
In this section, we prove our main result. Our main idea relies on the improvement of the lower bound of the solution of \eqref{E1.1A}. The first lower bound holds true for any positive supersolution. Similar result can also be found in \cite{DPQ25}.
We say $\Omega$ is an exterior domain if for some $R$ we have
$\{|x|\geq R\}\subset \Omega$.

\begin{lemma}\label{L1.2}
Let $\Omega$ be an exterior domain, $u:\Rn\to (0, \infty)$ be lower semicontinuous and solve 
$$
(-\Delta_p)^s u(x) \geq 0\quad  \text{ in }\;\; \Omega.$$
Then, for $N>ps$ and any positive $\theta\in (\frac{N-sp}{p-1},\frac{N}{p-1})$, there exists a constant $\kappa > 0$ such that 
$$u(x) \geq \kappa \min\{1, |x|^{-\theta}\}\quad  \text{ in }\; x\in\Rn.$$
\end{lemma}

\begin{proof}
From Lemma~\ref{DPQ3} we know $(-\Delta_p)^s \varphi_{\theta} (x) < 0$ in $B^c_1$. Choose $R>1$ large enough so that
$\{|x|\geq R\}\subset \Omega$.
Define $\kappa = \varepsilon_0^{\theta} \min_{|x| \leq R} u$. 
From the lower semicontinuity of $u$ it follows that $\kappa>0$. For any $\varepsilon >0$, we define 
$$
v_{\varepsilon} (x) := \kappa \varphi_{\theta} (x) - \varepsilon.
$$
It is easy to see that $(-\Delta_p)^s v_{\varepsilon} (x) < 0$ in 
$B^c_R$. Since $\theta > 0$, there exists $R_{\varepsilon}>R$ large enough so that $v_{\varepsilon} \leq 0$ for $|x| \geq R_{\varepsilon}$. For$|x| \leq R$, $v_{\varepsilon}(x) < \varepsilon_0^{\theta} (\min_{|x| \leq R} u) \varphi_{\theta} (x) \leq \varepsilon_0^{\theta} (\min_{|x| \leq R} u) \varepsilon_0^{-\theta} = \min_{|x| \leq R} u \leq u(x)$. Therefore,
since 
$$(-\Delta_p)^s v_\varepsilon<0\quad \text{and}\quad
(-\Delta_p)^s u(x)\geq 0\quad \text{in}\;\; B_{R_{\varepsilon}} \setminus B_R,$$
we have from Lemma~\ref{comparison} that $v_\varepsilon \leq u$ in $\Rn$.  Letting $\varepsilon \rightarrow 0$, we obtain $u(x) \geq \kappa |x|^{-\theta}$ for $x\in\Rn$.
Hence the result.
\end{proof}

Next key lemma improves the lower bound of $u$.
\begin{lemma}\label{L1.3}
Suppose that $t, m \geq 0$, $N>sp$ and  $m$ satisfies \eqref{ineq:m}.
 In addition, let
\begin{equation}\label{L3.2A}
t (N-sp) + m(N-(sp-p+1)) < N(p-1).
\end{equation}
Let $\Omega$ be an exterior domain, $u:\Rn\to (0, \infty)$ be lower semicontinuous that solve
\begin{equation}\label{L3.2B}
(-\Delta_p)^su \geq u^{t} |\na u|^{m}\quad
\text{in}\;\; \Omega.
\end{equation}
Then, for any $\theta \in \left( 0, \frac{N}{p-1} \right)$, there exists a constant $C(\theta, \Omega, u)$ such that 
\begin{equation}\label{eq1.2}
u(x) \geq C |x|^{-\theta}\quad  \text{ in }\;\; \bar \Omega \cap \{|x| \geq 1 \}.
\end{equation}
\end{lemma}

\begin{proof}
Using Lemma~\ref{L1.2} estimate \eqref{eq1.2} holds for any $\theta \in \left( \frac{N-sp}{p-1} , \frac{N}{p-1} \right)$,  hence it is enough to prove \eqref{eq1.2} only for $\theta \in \left( 0, \frac{N-sp}{p-1} \right]$. To this aim, we use an iteration process which we describe next.

Let $0<\sigma<\frac{N}{p-1}$  and suppose that
for some constant $C_{\sigma}$
we have
\begin{align}\label{L3.2D0}
u(x) &\geq C_{\sigma} |x|^{-\sigma}\quad \text{in}\; \bar\Omega\cap\{|x|\geq 1\}.
\end{align}
This clearly holds for $\sigma\in\left( \frac{N-sp}{p-1} , \frac{N}{p-1} \right)$, by Lemma~\ref{L1.2}. Using \eqref{L3.2D0}, we then have from \eqref{L3.2B} that
\begin{equation}\label{L5.1D}
(-\Delta_p)^s u \geq C_\sigma^t |x|^{-t\sigma}|\nabla u|^m \quad \text{in}\;\; \Omega\cap\{|x|\geq 1\}.
\end{equation}
Now, consider $\sigma_1$ such that $0<\sigma_1<\sigma$.
We claim that, for $\tilde\varphi_{\sigma_1}(x):= c \varphi_{\sigma_1}(x)$, where $c$ is a positive constant and $\varphi_{\sigma_1}$ is defined as in Lemma~\ref{DPQ3}, if we have
\begin{align}\label{L3.2D}
(-\Delta_p)^s \tilde\varphi_{\sigma_1} (x) &<  C_{\sigma}^t |x|^{-t \sigma} |\na \tilde\varphi_{\sigma_1}|^{m} \text{ for } |x| \geq r_{\sigma_1}
\end{align}
and some $r_{\sigma_1}>1$, then for some constant $C_{\sigma_1}$ we have
\begin{equation}\label{eqclaim2}
u(x) \geq C_{\sigma_1} |x|^{-\sigma_1} \text{ in } \bar \Omega \cap \{|x| \geq 1 \}.
\end{equation}

To prove this claim, since $u$ is positive in $\bar\Om$ and also lower semicontinuous, it is enough to prove \eqref{eqclaim2} for $|x|\geq r_{\sigma_1}$. We also set $r_{\sigma_1}$ large enough so that 
$\{|x|\geq r_{\sigma_1}\} \subset \Omega$.

Since $m\leq p-1$, from \eqref{L3.2D} we get that for $\kappa \in (0,1)$
$$
(-\Delta_p)^s (\kappa \tilde \varphi_{\sigma_1})(x) 
=\kappa^{p-1} (-\Delta_s)^p \tilde\varphi_{\sigma_1} (x) 
< \kappa^{m}  C_{\sigma}^t |x|^{-t \sigma} |\Grad \tilde\varphi_{\sigma_1}|^{m} 
=   C_{\sigma}^t |x|^{-t \sigma} |\Grad (\kappa \tilde\varphi_{\sigma_1})|^{m} $$
for  $|x| \geq r_{\sigma_1}$.
Choose $\kappa$ small enough so that $u(x) \geq \kappa {\tilde\varphi_{\sigma_1}}$ in $|x| \leq r_{\sigma_1}$. Let $\varepsilon \in (0,1)$ be small and define 
$$\psi(x) := \kappa{\tilde \varphi_{\sigma_1}} (x) -\varepsilon.$$
Therefore, 
\begin{equation*}
(-\Delta_p)^s \psi(x)  <   C_{\sigma}^t |x|^{{-t \sigma}} |\Grad \psi(x)|^{m}\quad  \text{ for } |x| \geq r_{\sigma_1}.
\end{equation*}
Thus $\psi$ is a strict subsolution of $(-\Delta_p)^s v=C_{\sigma}^t |x|^{{-t \sigma}} |\Grad v(x)|^{m}$ in $\{|x| \geq r_{\sigma_1}\}$ and by \eqref{L5.1D} $u$ is a supersolution of the same equation.  Now we can apply an argument similar to Lemma~\ref{L1.2} in conjunction with Lemma~\ref{comparison} to obtain
that $\psi\leq u$ in $\Rn$. Letting $\varepsilon\to 0$ we get
\eqref{eqclaim2}. Thus we have proved that for any $0<\sigma_1<\sigma$ if \eqref{L3.2D0} holds then \eqref{eqclaim2} also holds provided \eqref{L3.2D} is satisfied. We refer to this as {\it $\sigma_1$ is an improvement over $\sigma$.}

Next we show that given any $\theta\in (0, \frac{N-sp}{p-1}]$, we can find a finite number of pairs $(\sigma_{i+1}, \sigma_{i})$, 
$i=0,1, 2, \ldots, k-1$, such that $\sigma_0>\frac{N-sp}{p-1}$,
${ \sigma_{k}}\leq \theta$ and $\sigma_{i+1}$ is an improvement over $\sigma_{i}$. It is evident that this will prove \eqref{eq1.2}.
Let $\sigma_0=\frac{N-sp}{p-1} + \delta$ for some 
$\delta > 0$ small. By Lemma~\ref{L1.2} there exists
$C_{\sigma_0}$ satisfying
\begin{equation*}
u(x) \geq C_{\sigma_0} |x|^{- \sigma_0} \text{ in } \bar \Omega \cap \{ |x| \geq 1\}.
\end{equation*} 
Let $\sigma_1 = \frac{N-sp}{p-1}<\sigma_0$.
Also, recall the function $\varphi_\sigma$ from Lemma~\ref{DPQ3}.
Then, from Lemma~\ref{DPQ3} we have
\begin{align*}
(-\Delta_p)^s \varphi_{\sigma_1}(x) &\leq C |x|^{-\sigma_1(p-1)-sp},
\\
|x|^{-t \sigma_0}|\na \varphi_{\sigma_1}(x)|^{m} &=
\sigma^m_1\, |x|^{-t \sigma_1 -(\sigma_1+1)m -t \delta},
\end{align*}
for $|x|\geq 1$.
Using \eqref{L3.2A} we can choose $\delta >0$ small enough so that
$$N(p-1) > m\big(N-(sp-p+1)\big)  +t(N-sp) +\delta t(p-1)\Leftrightarrow
\sigma_1(p-1)+ sp > (\sigma_1+1) m + t \sigma_1 +t \delta.$$

Therefore, for some large enough $r_{\sigma_1}$ we have 
$$
(-\Delta_p)^s \varphi_{\sigma_1}(x) <  C_{\sigma_0}^t |x|^{-t \sigma_0}|\na \varphi_{\sigma_1}|^{m} \quad \text{ in }\;  |x| \geq r_{\sigma_1}.$$
Thus $\sigma_1$ is an improvement over $\sigma_0$ and therefore, \eqref{eq1.2} holds for $\theta=\sigma_1=\frac{N-sp}{p-1}$.

Now we complete the rest of the proof for $0<\theta < \sigma_1$ using an iteration method. 
Let $0<\sigma_{i+1}<\sigma_i\leq \sigma_1$.
From  Lemma~\ref{DPQ3} we have $(-\Delta_p)^s \varphi_{\sigma_{i}}(x) \leq C_{i} |x|^{-\sigma_{i}(p-1)-sp}$ for 
$\sigma_{i} \in \left(0, \frac{N-sp}{p-1} \right]$. Again, from the definition of $\varphi_{\sigma_i}$, we also have
$$|x|^{-t \sigma_i} |\na \varphi_{\sigma_{i+1}}|^m = \sigma^m_{i+1}|x|^{-t \sigma_i - (\sigma_{i+1}+1)m}\quad \text{for}\;\; |x|> 1.$$
Note that for  $|x|>1$ large enough
$$
C_{i+1}|x|^{-\sigma_{i+1}(p-1)-sp} < C^t_{\sigma_{i}} \sigma_{i+1}^m |x|^{-t \sigma_i - (\sigma_{i+1}+1)m}
$$
will hold, provided 
\begin{align}\label{eqclaim1}
   (\sigma_{i+1}+1)m+ t \sigma_i <\sigma_{i+1} (p-1)+sp=(\sigma_{i+1}+1)(p-1)+(sp-p+1),
\end{align}

for $i=1,2, 3,\ldots$, where $C_{\sigma_{i}}$ is the constant appearing in \eqref{L3.2D0} for $\sigma=\sigma_i$.
Therefore, it is enough to construct pair $(\sigma_{i+1}, \sigma_i)$ satisfying \eqref{eqclaim1} for $i=1,2, \ldots$.
It is also important to keep in mind that \eqref{L3.2D0} holds for $\sigma=\sigma_1$, as already shown above.
\medskip

\noindent {Case 1.} Suppose $m=p-1$.
Using \eqref{L3.2A} we have 
$$t (N-sp) < m(sp-p+1) \Rightarrow t \frac{N-sp}{p-1} < sp-p+1 \Rightarrow t \sigma< sp-p+1  $$
for any $\sigma \leq \frac{N-sp}{p-1}$, which implies \eqref{eqclaim1} as $m=p-1$. Thus, if we set $\sigma_2=\theta$, then
$\sigma_2$ is an improvement over $\sigma_1$.

\medskip

\noindent{Case 2.} Suppose that $m<p-1$ and $t \frac{N-sp}{p-1} \leq sp -m$. In this case we have for any 
$\sigma_2<\sigma_1\leq \frac{N-ps}{p-1}$
\begin{align*}
(\sigma_2+1)m + t \sigma_1 \, \leq \, (\sigma_2+1)m + t \frac{N-sp}{p-1}
&\leq (\sigma_1+1)m + sp-m
\\
&\leq \sigma_1 m + sp
\\
&< \sigma_1(p-1) +sp
\\
&=(\sigma_1 + 1)(p-1) +sp-p+1.
\end{align*}
Therefore, letting $\sigma_2=\theta$ we can complete the proof as before.
\medskip

\noindent{Case 3.} Suppose that $m<p-1$ and $t  \frac{N-sp}{p-1} > sp -m$. Using \eqref{ineq:m}, it follows that $m<sp$ which, in turn, implies $t>0$ in this case.
Choose $\epsilon \in (0, sp-m)$ small enough such that $t(N-sp) + m(N-sp+p-1) +\epsilon(p-1) < N(p-1)$. Such an $\epsilon$ exists due to \eqref{L3.2A}.
Define the function
$$\mathfrak{g} (\ell) =\frac{t \ell -sp+p-1+\epsilon}{p-1-m} - 1.$$
Note that $\mathfrak{g} (\ell) >0 $ whenever $t \ell + \epsilon > sp-m$. Define $\sigma_1 =\frac{N-sp}{p-1}$ and $\sigma_i = \mathfrak{g} (\sigma_{i-1})$ for $i=2,3, \ldots.$

\begin{align*}
\mathfrak{g} (\sigma_1) < \sigma_1 &\Leftrightarrow \frac{t \sigma_1 -sp+p-1+\epsilon}{p-1-m} < \frac{N-sp+p-1}{p-1}
\\
&\Leftrightarrow  t \sigma_1 -sp+p-1 + \epsilon<\frac{N-sp+p-1}{p-1} (p-1-m)
\\
&\Leftrightarrow t \frac{N-sp}{p-1} + \epsilon < N-sp +p-1 -\frac{m}{p-1}(N-sp+p-1) +sp -p+1
\\
&\Leftrightarrow t \frac{N-sp}{p-1} + \epsilon< N - \frac{m}{p-1}(N-sp+p-1)
\\
&\Leftrightarrow t(N-sp) +m(N-sp+p-1) + \epsilon(p-1) < N(p-1)
\end{align*}
and the last inequality is ensured by our choice of $\epsilon$. Therefore, $\sigma_2 < {\sigma_1}$. Again, using the linearity of $\mathfrak{g}$ we have 
$$\sigma_{i+1}-\sigma_i=\left(\frac{t}{p-1-m}\right)^{i-1}(\sigma_2 - \sigma_1)<0\quad \text{for}\;\; i=2, 3, \ldots.$$ 
Thus, $\{\sigma_i\}$ forms a strictly decreasing sequence. Also,
$$(\sigma_{i+1} +1)(p-1)+sp-p+1= t \sigma_i + m(\sigma_{i+1} +1) + \epsilon > t \sigma_i + m(\sigma_{i+1} +1).$$ 
Hence, \eqref{eqclaim1} holds for the pair $(\sigma_{i+1}, \sigma_i)$ with $\sigma_{i+1} = \mathfrak{g}(\sigma_i)$, which also
confirms $\sigma_{i+1}$ as an improvement over $\sigma_i$ for $i=1,2,\ldots$.

To complete the proof, we only need to show that there exists $n\geq 2$ such that $t\sigma_n\leq sp-m$ and \eqref{L3.2D0} holds for $\sigma=\sigma_n (<\sigma_1=\frac{N-sp}{p-1})$. Indeed, once we have such a  $\sigma_n$,
we can repeat the argument of Case 2 with the choice of $\sigma_{n+1}\leq \theta$.
We claim that there exists ${n}\geq 2$ such that $t { \sigma_n} <sp-m$ and then \eqref{eqclaim1} holds for the pair
{ $(\sigma_{n},\sigma_{n-1})$}.  If $\frac{t}{p-1-m}<1$ then $\mathfrak{g}$ is a contraction map with a negative fixed point 
$-\frac{sp-m-\epsilon}{p-1-m-t}$, hence there exists ${n}\geq 1$ such that $t \sigma_n \in (0, sp-m]$ and $t \sigma_{{ n}-1}>sp-m$. If $\frac{t}{p-1-m}\geq 1$, then $\sigma_n \rightarrow  -\infty$ as $n \rightarrow \infty$, and therefore, there exists ${n}\geq 1$ such that $t \sigma_{n} \in (0, sp-m]$ and $t \sigma_{n-1}>sp-m$. This proves our claim and completes the proof.
\end{proof}

%

Now we are ready to provide a proof of Theorem~\ref{gradnon}.
\begin{proof}[Proof of Theorem \ref{gradnon}]
In view of Lemma~\ref{maxpr}, we only need to consider the case
$u>0$ in $\Rn$. Otherwise, $u\equiv 0$ and the proof is done.
 Choose $\theta > 0$ small enough and 
$\bar \theta \in (0,\frac{sp}{p-1}\wedge 1)$ so that 
\begin{itemize}
\item[(i)] If $sp> p-1$, then $t \theta < sp-p+1$.

\item[(ii)] If $sp=p-1$,  then $ t \theta < (1-\bar \theta) (p-1-m)$. This is possible to do since $m<sp$ in this case by our 
assumption on $m$.

\item[(iii)] If $sp<p-1$, then choose $\bar \theta$ small enough 
(we actually need $\bar \theta <{\frac{sp-m}{p-1-m}}$ which is less than $\frac{sp}{p-1} $)
so that $\bar \theta(p-1-m) -(sp-m) + t \theta<0$.
\end{itemize}
In all the three cases, we observe that 
\begin{equation}\label{eqneg}
(\bar\theta-1)(p-1-m) +p-1-sp + t \theta  <0.
\end{equation}
For $\varepsilon \in (0,1)$, define $\varphi_{\varepsilon} (x) := -\varepsilon|x|^{\bar \theta}$.

We next  claim that for $\upkappa >0$ there exists $r_{\upkappa} > 0$ satisfying
\begin{equation}\label{claim}
(-\Delta_p)^s \varphi_{\varepsilon}(x) < \upkappa\, |x|^{-t \theta} |\Grad \varphi_{\varepsilon}|^m\quad  \text{ for }\;\; |x|\geq r_{\kappa}.
\end{equation}
To prove the claim,  consider $x \neq 0$. Using Theorem \ref{DPQ1} we get
\begin{align*}
(-\Delta_p)^s \varphi_{\varepsilon}(x) &= -\varepsilon^{p-1}C(\bar \theta) |x|^{\bar \theta (p-1)-sp} 
\\
&\leq \varepsilon^{m} |C(\bar \theta)|\,  |x|^{(\bar\theta-1)(p-1-m) +p-1-sp + t \theta +[(\bar \theta -1)m -t \theta]}
\\
&= |C(\bar \theta)| \, |x|^{(\bar\theta-1)(p-1-m) +p-1-sp + t \theta } (\varepsilon|x|^{\bar \theta -1})^m|x|^{-t \theta}
\\
&=  |C(\bar \theta)|\bar\theta^{-m}\, |x|^{(\bar\theta-1)(p-1-m) +p-1-sp + t \theta } |x|^{-t \theta} |\na \varphi_{\varepsilon}|^m.
\end{align*}
From \eqref{eqneg} we see that $|x|^{(\bar\theta-1)(p-1-m) +p-1-sp + t \theta } < \frac{\upkappa\bar\theta^m}{|C(\bar \theta)|}$ for $|x|$ large enough. Therefore, there exists $r_{\kappa} >0$ large enough so that \eqref{claim} holds.

Now, let $u$ be any positive solution of $(\Delta_p)^s u \geq u^{t} |\Grad u|^m \text{ in } \Rn$. Applying Lemma \ref{L1.3} we see that 
$$(-\Delta_p)^s u \geq \upkappa |x|^{-t \theta} |\na u|^m \quad \text{ for }\;\; |x| >1 \text{ and some } \kappa>0.$$
We use this $\upkappa$ in \eqref{claim} and adjust $r_{\kappa}$ accordingly. Next  we show that
\begin{equation}\label{T5.1D}
u(x) \geq \min_{\bar B_{r_{\upkappa}}} u := \rho_{\upkappa} \text{ in } \Rn.
\end{equation}
To establish \eqref{T5.1D} we define $\tilde \varphi_{\varepsilon} := \rho_{\kappa} + \varphi_{\varepsilon}$. Using \eqref{claim} and
applying the comparison argument as in Lemma~\ref{L1.2} we obtain $u\geq \tilde \varphi_{\varepsilon}$ in 
$\Rn$ for all $\varepsilon>0$. Letting $\varepsilon \rightarrow 0$, this gives us $u(x) \geq \rho_{\kappa}$
in $\Rn$. Thus $u$ attains its minimum in the ball 
$\bar{B}_{r_\upkappa}$. We note that $v:=u-\rho_\upkappa$ also 
solves $(-\Delta_p)^s v\geq 0$ in $\Rn$.
Using Lemma \ref{maxpr} it then follows that $u \equiv \rho_{\upkappa}$, completing the proof. 
\end{proof}

\medskip

{\bf Funding:} This research of M.~Bhakta
is partially supported by a DST Swarnajaynti fellowship (SB/SJF/2021-22/09). A.~Biswas is partially supported by a DST Swarnajaynti fellowship (SB/SJF/2020-21/03).

\bigskip

{\bf Data availability:} Data sharing not applicable to this article as no datasets were generated or analyzed during the current study.

\medskip

{\bf Conflict of interest} The authors have no Conflict of interest to declare that are relevant to the content of this article.

\bigskip

\bibliographystyle{plain}
\bibliography{BBS}

\end{document}